\documentclass[11pt,a4paper]{article}

\usepackage[left=2.45cm, top=2.45cm,bottom=2.45cm,right=2.45cm]{geometry}

\usepackage{amsmath,amssymb,amsthm,mathrsfs,calc,graphicx,stmaryrd,xcolor,enumerate,cleveref,dsfont}
\usepackage[british]{babel}
\usepackage{amsfonts}              

\newtheoremstyle{definition}
{10pt}
{10pt}
{}
{}
{\bfseries}
{}
{.5em}
{}

\newtheoremstyle{plain}
{10pt}
{10pt}
{\itshape}
{}
{\bfseries}
{}
{.5em}
{}

\theoremstyle{plain}	
\newtheorem{Theorem}{Theorem}[section]
\newtheorem{Lemma}[Theorem]{Lemma}
\newtheorem{Proposition}[Theorem]{Proposition}

\theoremstyle{definition}	
\newtheorem{Definition}[Theorem]{Definition}
\newtheorem{Remark}[Theorem]{Remark}

\newcommand{\R}{\mathbb{R}}
\newcommand{\N}{\mathbb{N}}
\newcommand{\Prob}{\mathbb{P}}
\newcommand{\E}{\mathbb{E}}
\newcommand{\Sph}{\mathbb{S}}

\newcommand{\diff}{\textup{d}}

\DeclareMathOperator{\conv}{conv}

\DeclareMathOperator{\Vol}{Vol}

\DeclareMathOperator{\Lip}{Lip}
\DeclareMathOperator{\dVar}{V}

\DeclareMathOperator{\Per}{Per}
\DeclareMathOperator{\dive}{div}

\allowdisplaybreaks

\begin{document}
	
	\title{\bfseries Moments of the maximal number of empty simplices of a random point set}
	
	\author{Daniel Temesvari\footnotemark[1]}
	
	\date{}
	\renewcommand{\thefootnote}{\fnsymbol{footnote}}
	\footnotetext[1]{Ruhr University Bochum, Faculty of Mathematics, D-44780 Bochum, Germany. E-mail: daniel.temesvari@rub.de}

	\maketitle
	
	\begin{abstract}
		For a finite set $X$ of $n$ points from $ \R^M$, the degree of an $M$-element subset $\{x_1,\dots,x_M\}$ of $X$ is defined as the number of $M$-simplices that can be constructed from this $M$-element subset using an additional point $z \in X$, such that no further point of $X$ lies in the interior of this $M$-simplex.
		Furthermore, the degree of $X$, denoted by $\deg  (X)$, is the maximal degree of any of its $M$-element subsets.
		
		The purpose of this paper is to show that the moments of the degree of $X$ satisfy $\E[ \deg (X)^k ] \geq c n^k/\log n$, for some constant $c>0$, if the elements of the set $X$ are chosen uniformly and independently from a convex body $W \subset \R^M$. Additionally, it will be shown that $\deg (X)$ converges in probability to infinity as the number of points of the set $X$ goes to infinity.
	\end{abstract}
	
	\section{Introduction}
	
	Let $M\in \N$, with $M\geq 2$, and consider a finite set $X$ of points from $\R^M$ in general position, i.e., no $M+1$ points of $X$ lie in an affine hyperplane of $\R^M$. For an $M$-element subset $\{x_1,\dots,x_M\}$ of $X$, one says that the $M$-simplex $\conv(\{x_1,\dots,x_M,z\})$, $z \in X\setminus \{x_1,\dots,x_M\}$, is empty, if there lies no further point of $X$ in the convex hull of the points $x_1,\dots,x_M,z$, i.e., if $\conv(\{x_1,\dots,x_M,z\})\cap X = \{x_1,\dots,x_M,z\}$. The number of empty simplices that can be formed from an $M$-element subset $\{x_1,\dots,x_M\}\subset X$ is called the degree of $\{x_1,\dots,x_M\}$ and is denoted by $\deg(x_1,\dots,x_M)$. The degree of the set $X$, written as $\deg(X)$, is introduced as the maximum of the degrees $\deg(x_1,\dots,x_M)$ over all $M$-element subsets $\{x_1,\dots,x_M\}\subset X$.
	
	This paper is concerned with the question, whether the degree of a finite set $X\subset \R^M$, $|X|=n$, tends to infinity as $n$ goes to infinity, or, in a more descriptive language, if the maximal number of empty simplices one can obtain for any of the $M$-element subsets of $X$ grows proportionally to the number of points of the set $X$.
	
	This problem dates back to the early nineties of the last millennium as P. Erd\H{o}s \cite{E} raised the question, whether, in the $2$-dimensional case, the degree of such a set $X$ goes to infinity as the number of points of $X$ goes to infinity, i.e., if the maximal number of empty triangles for any pair of points of the set $X$ in the plane goes to infinity as the number of points of the set $X$ go to infinity.
	It was then conjectured by I. B\'{a}r\'{a}ny that this is indeed true.
	This conjecture was later repeated in \cite{BK}, as well as in \cite{BMP}. Although I. B\'{a}r\'{a}ny and Gy. K\'{a}rolyi \cite{BK} showed that $\deg(X)\geq 10$ for sufficiently large $n$ and I. B\'{a}r\'{a}ny and P. Valtr \cite{BV} constructed a set $X$ in general position such that $\deg(X)=4\sqrt{n}(1+o(1))$, it is still unknown whether this conjecture is true.
	
	However, I. B\'{a}r\'{a}ny, J.F. Marckert and M. Reitzner showed in \cite{BMR}, that $\E[\deg(X)] \geq cn/\log n$, for some constant $c$, if the points of the set $X$ are chosen uniformly and independently from a convex body $W\subset \R^2$, i.e., a compact, convex set with nonempty interior. Furthermore, they also showed that in this setting the degree of $X$ converges in probability to infinity as $n$ goes to infinity.
	
	The purpose of this article is to prove the following generalization:
	
	\begin{Theorem}\label{t8}
		Let $M \in \N$, with $M\geq2$, and $X \subset \R^M$ be a set of $n$ points, $n$ sufficiently large, chosen uniformly and independently from a convex body $W \subset \R^M$. Then, for all $k \in \N$,
		\begin{enumerate}[(i)]
			\item
				$\E[\deg(X)^k] \geq cn^k/\log n$, for some positive constant $c>0$, and
			\item
			   $ \deg(X) \rightarrow \infty$ in probability, as $n\rightarrow \infty$.
		\end{enumerate}
	\end{Theorem}
	This theorem generalizes the results from \cite{BMR} in two ways: First, it shows that the lower bound for $\E[\deg(X)]$ holds true in arbitrary dimensions, and, second, it also gives similar lower bounds to the higher moments of $\deg(X)$, i.e., that for any dimension $M$, natural number $k\in \N$ and large enough set $X\subset \R^M$ of random points, chosen uniformly and independently from a convex body $W$, the lower bound $\E[\deg(X)^k]\geq c n^k / \log n$, for a constant $c>0$, holds for the $k$-th moment of the degree of $X$. Furthermore, it shows the convergence in probability of the degree of $X$ to infinity as the number of points grows to infinity. 
	
	Note that a lower bound for the moments comes up quite naturally by simply applying Jensen's inequality to the expectation of the $k$-th power of the positive random variable $\deg(X)$, which would yield $\E[\deg(X)^k]\geq cn^k / (\log n)^k$. However, the result of \Cref{t8} provides a compelling improvement of this trivial bound.

	We would like to highlight two significant differences of the proof of the main theorem from \cite{BMR} and the proof presented in this paper. The authors of \cite{BMR} make use of two results from \cite{RST}, namely, Theorem 3.1 and Theorem 6.6, which both are elaborated in arbitrary dimension, but only for pairs of points instead of $M$-element subsets. Although both of the theorems can be generalized for our setting, it turns out that the generalization of Theorem 6.6. is not powerful enough to be used in the main proof.
	
	Theorem 3.1 from \cite{RST} relies on conclusions from \cite{G} on the covariogram of a set, relating finiteness of the perimeter of a set to the Lipschitz continuity of the covariogram of the set, as well as the existence of all the directional right derivatives in the point zero for the covariogram of the set.  We generalize this result for $M$-element subsets by generalizing the notion of the covariogram of a set. Although we cannot obtain a full equivalent of the results from \cite{G} for the generalized covariogram, it still suffices for our purposes (see \Cref{t4} and \Cref{t6}).
	
	The second result used from \cite{RST}, namely, Theorem 6.6, is a certain concentration inequality, which can also be generalized for $M$-element subsets, but, as already mentioned, this generalization loses its power on this way. In particular, Theorem 6.6. yields a polynomially decaying upper bound to the probability of a certain functional $N_T(X)$, defined in the preliminaries, being bigger than some $\beta\in \R^+$. It heavily relies on the fact that the number of pairs of points of $X$ grows asymptotically like $n^2$. However, it turns out that using Markov's inequality as a substitute suffices and even simplifies the already existing proof for the expectation in the $2$-dimensional case.
	
	\section{Preliminaries}

	\subsection{General preliminaries}

	First, the basic definitions are restated once more. Let $M\in \N$, with $M\geq 2$, and consider $X$ to be a finite set of points in $\R^M$ in general position, i.e., no $M+1$ points of $X$ lie in an affine hyperplane of $\R^M$. By ${X \brack k}$ we denote the set of all $k$-element subsets of $X$.
	An $M$-simplex $\conv(\{x_1,\dots,x_{M+1}\})$ with $\{x_1,\dots,x_{M+1}\} \in {X \brack M+1}$ is said to be empty (in $X$) if $\conv(\{x_1,\dots,x_{M+1}\})\cap X = \{x_1,\dots,x_{M+1}\}$.
	The degree of $\{x_1,\dots,x_M\} \in {X \brack M}$ is the number of points $y \in X$ such that $\{x_1,\dots,x_M,y\} \in {X \brack M+1}$ determines an empty simplex and is denoted by $\deg(x_1,\dots,x_M)$ or $\deg(x_1,\dots,x_M;X)$, if an emphasis is put on the set $X$. We set
	\[ \deg (X) := \max \left\{ \deg(x_1,\dots,x_M): \{x_1,\dots,x_M\} \in {X \brack M} \right\} \mbox{.} \]
	It immediately follows that $\deg(X)\leq n-M$.

	By $\Vol_M(A)$ we mean the $M$-dimensional Lebesgue volume of a set $A$, and, specifically, by $\Vol_M(x_1,\dots,x_{M+1})$ we mean the $M$-dimensional Lebesgue volume of the $M$-simplex spanned by $x_1,\dots,x_{M+1}$.
	The symbol $\mathbb{B}^M_T(x)$ will be used to denote the $M$-dimensional Euclidean ball with radius $T>0$ centered at $x \in \R^M$. The unit ball $\mathbb{B}^M_1(0)$ will be referred to as $\mathbb{B}^M$ and its unit sphere as $\Sph^{M-1}$. The $M$-dimensional Lebesgue volume of the unit ball will be abbreviated as $\kappa_{M}:=\Vol_M\left(\mathbb{B}^M\right)$, while for the $(M-1)$-dimensional normalized Hausdorff measure of the unit sphere the symbol $\omega_{M}:=\mathcal{H}^{M-1}\left(\Sph^{M-1}\right)$ is used. Note that $\omega_{M}=M\kappa_M$.
	
	Additionally, we define for $T>0$ and $\mathcal{M}=\{1,\dots,M\}$ the functionals
	\[ N_T(X) := \sum_{ \{x_1,\dots,x_M\} \in {X \brack M} } \mathds{1} \left( \exists i \in \mathcal{M}: \{x_1,\dots,x_M\} \subset \mathbb{B}^M_{T}\left(x_i\right) \right)  \]
	and
	\[
	F^{(k)}_T(X) := \sum_{\{x_1,\dots,x_M\} \in {X \brack M}} \mathds{1} \left( \exists i \in \mathcal{M}: \{x_1,\dots,x_M\} \subset \mathbb{B}^M_{T}\left(x_i\right) \right) \deg(x_1,\dots,x_M;X)^k
	\]
	for all $k \in \N$.
	
	The purpose of the first functional $N_T(X)$ is to count the number of $M$-element subsets $\{x_1,\dots,x_M\} \subset X$ for which an $x_i$, $i \in \mathcal{M}$, can be found, such that all the other points from that subset are not further than $T$ away from $x_i$. The functionals $F^{(k)}_T(X)$ do the same as $N_T(X)$, but, additionally, weight each of the counted subsets with the $k$-th power of its degree. These functionals will lie at the core of the proof of \Cref{t8}.
	
	From now on $\xi_n$ will be used to denote a set of $n$ random points chosen uniformly and independently from a convex body $W\subset\R^M$, i.e., a  compact convex set with nonempty interior. Note, that with probability one, such a set $\xi_n$ will be in general position.
	
	An essential tool that will be needed is the binomial counterpart of Mecke's formula for Poisson processes. For a fixed integer $k\geq 1$ and a non-negative measurable function $f:W^k \rightarrow \R$ this reads as
	
	\begin{equation}\label{e7}
	\E \sum_{\{x_1,\dots,x_k\}\in {\xi_n \brack k}} f\left(x_1,\dots,x_k\right) = \binom{n}{k} \int \limits_{W} \dots\int \limits_{W} f\left(x_1,\dots,x_k\right) \diff x_1\dots\diff x_k \mbox{.}
	\end{equation}

	Moreover, the Landau notation is used. Let $f,g:\R\rightarrow\R$ and $a \in \R$. Then, $g=o(h)$ as $t\to a$, if $\lim_{t\to a} |g(t)/h(t)|=0$ and $g=\mathcal{O}(h)$ as $t\to a$, if $\limsup_{t\to a} |g(t)/h(t)|<\infty$.
	Throughout this paper constants will be denoted by $c$ and may vary from instance to instance.
		
	\subsection{Preliminaries concerning the covariogram}
	
	The covariogram of a Lebesgue measurable set $W\subset \R^M$ is defined as the mapping
	\[ g_W: \R^M \rightarrow [0,\infty) \mbox{,}~ y \mapsto \Vol_M(W\cap(y+W)) \mbox{.} \]
	In order to be able to derive our main result, a generalized notion of the covariogram of a Lebesgue measurable set $W\subset \R^M$ has to be considered.
	
	\begin{Definition}\label{t1}
		Let $W \in (\R^M)^{M-1}$ be a Lebesgue measurable set of finite Lebesgue measure. Define the generalized covariogram of $W$ as $g_W: (\R^M)^{M-1} \rightarrow [0,\infty)$, for all $y=(y_1,\dots,y_{M-1}) \in (\R^M)^{M-1}$, by
		\[ g_W(y)=\Vol_M \left(W \cap (y_1 + W) \cap \dots \cap (y_{M-1} + W) \right) \mbox{.} \]
	\end{Definition}
	
	\begin{Remark}\label{r2}
		Two things are worth to be noted here. First, the generalized covariogram can be written in terms of an integral over indicator functions of the set $W$, i.e.,
		\begin{align*}
		g_W(y)&=\Vol_M \left(W \cap (y_1 + W) \cap \dots \cap (y_{M-1} + W) \right) \\
		&= \int \limits_{\R^M} \mathds{1} (x \in W, x-y_1 \in W,\dots, x-y_{M-1} \in W)\, \diff x \\
		&= \int \limits_{\R^M} \mathds{1}_W(x)\prod_{i=1}^{M-1} \mathds{1}_W(x-y_i)\, \diff x
		\mbox{,}
		\end{align*}
		and, second, it is symmetric with respect to permutations of the vectors $y_1,\dots,y_{M-1}$, i.e.,
		\begin{equation}\label{eq:perm}
		g_W(y_1,\dots,y_{M-1}) = g_W\left(y_{\sigma(1)},\dots,y_{\sigma(M-1)}\right)
		\end{equation}
		holds for all permutations $\sigma \in \mathcal{S}_{M-1}$.
	\end{Remark}

	Let $U \subseteq \R^M$. Denote by $\mathcal{C}_c^1(U,\R)$ and $\mathcal{C}_c^1(U,\R^n)$ the sets of continuously differentiable functions from $U$ to $\R$ and $U$ to $\R^n$, respectively. Let $L_{loc}^1(U)$ be the sets of locally integrable functions over $U$ and $L^1(U)$ the set of integrable functions over $U$.\\	
	For a function $f \in L_{loc}^1(U)$ the variation in $U$ is defined as
	\[ V(f,U):=\sup \left\lbrace \int_U f(x) \dive \varphi(x) \diff x : \varphi \in \mathcal{C}_c^1\left(U,\R^M\right), \| \varphi \|_\infty \leq 1 \right\rbrace \mbox{,} \]
	while the directional variation in $U$ in the direction $u \in \Sph^{M-1}$ is defined as
	\[ V_u(f,U):=\sup \left\lbrace \int_U f(x) \frac{\partial \varphi}{\partial u} (x) \diff x : \varphi \in \mathcal{C}_c^1\left(U,\R\right), \| \varphi \|_\infty \leq 1 \right\rbrace \mbox{.} \]
	The perimeter of a set $W\subset\R^M$ in $U$ is introduced as $\Per(W,U):=V(\mathds{1}_W,U)$, i.e., the variation of the indicator function of $W$ in $U$. In the case that $U=\R^M$, one writes $\Per(W):=\Per(W,\R^M)$. Note that, if $W$ is convex, then $\Per(W)=\mathcal{H}^{M-1}(\partial W)$ holds. Analogously, the directional variation of the set $W$ in the direction $u \in \Sph^{M-1}$ is defined as $V_u(W,U):=V_u(\mathds{1}_W,U)$. Again, in the case that $U=\R^M$ one, writes $V_u(W):=V_u(W,\R^M)$. For further information on these topics see \cite{G} and the textbook of Ambrosio, Fusco and Pallara \cite{ACV}. 
	The following two propositions, which will be used in the next section, are taken from \cite{G}, where one may also find their proofs.
	The first proposition shows that bounded variation of a function $f$ in $U$ is equivalent to $f$ having bounded directional variation in $U$ for each direction. Additionally, it provides a way to calculate the variation by means of integration over all directional variations.  
	\begin{Proposition}\label{p1}
		Let $U\subseteq \R^M$ be open and let $f \in L^1(U)$. Then, the following statements are equivalent:
		\begin{enumerate}[(i)]
			\item
				$V(f,U)<\infty$,
			\item
				$V_u(f,U)<\infty$ for all $u \in \Sph^{M-1}$,
			\item
				$V_{e_i}(f,U)<\infty$ for all vectors $e_i$ of the canonical basis of $\R^M$.
		\end{enumerate}
		Additionally,
		\[ \frac{1}{M} V(f,U) \leq \frac{1}{M} \sum_{i=1}^{M} V_{e_i}(f,U) \leq \sup_{u \in \Sph^{M-1}} V_u(f,U) \leq V(f,U) \]
		and
		\[ V(f,U) = \frac{1}{2 \kappa_{M-1}} \int_{\Sph^{M-1}} V_u(f,U) \mathcal{H}^{M-1}(\diff u) \]
		holds.
	\end{Proposition}
	
	The second proposition elaborates a method to calculate directional variations of a function $f$ in $U$ by integrals of difference quotients.
	
	\begin{Proposition}\label{p2}
		Let $u \in \Sph^{M-1}$ and $f \in L^1(\R^M)$. Then,
		\[ \int_{\R^M} \frac{|f(x+ru)-f(x)|}{|r|}\, \diff x \leq V_u(f) \mbox{,} \]
		for all $r\neq 0$, and
		\[ \lim_{r \rightarrow 0}  \int_{\R^M} \frac{|f(x+ru)-f(x)|}{|r|}\, \diff x = V_u(f) \mbox{.} \]
	\end{Proposition}
	
	\section{Generalized covariogram}
	
	As already mentioned in the introduction, this section is based upon the results of \cite{G}, where it is shown that for a measurable set the three properties of having finite perimeter, Lipschitz continuous covariogram and existent directional right derivatives of the covariogram in zero, are equivalent. Furthermore, the perimeter can be computed as an integral of all these directional right derivatives over the unit sphere.
	A similar result, namely, \Cref{t6}, is provided here for the generalized covariogram.

	Two preliminary lemmas on basic properties of the generalized covariogram are necessary. The first one is concerned with an upper bound on the absolute value of the distance of two points of the covariogram, whereas the second one addresses the issue of finding a representation for this upper bound in terms of an integral of indicator functions.
	
	These lemmas will be used to show \Cref{t4} and \Cref{t6}. The first proposition provides statements for a measurable set in terms of certain directional derivatives of the covariogram and Lipschitz continuity of a certain restriction of the covariogram, which are equivalent to the property of the measurable set having finite directional variation in a direction $u \in \Sph^{M-1}$. Additionally, a formula for calculating this finite directional variation in a direction $u \in \Sph^{M-1}$ is given. The second proposition relates the property of having finite perimeter to the existence of right derivatives in zero of the aforementioned restriction of the covariogram, for any direction $u \in \Sph^{M-1}$, and allows to calculate the perimeter by integrating these derivatives over $\Sph^{M-1}$.
	
	\begin{Lemma}\label{t2}
		Let $W\subset\R^M$ be Lebesgue measurable and let $g_W$ be its generalized covariogram. Let $\tilde{y},\tilde{z} \in \R^M$. Define $y,z\in \left(\R^M\right)^{M-1}$ by $y:=(\tilde{y},0,\dots,0)$ and $z:=(\tilde{z},0,\dots,0)$. Then,
		\[ |g_W(y)-g_W(z)|\leq g_W(0,\dots,0)-g_W(y-z) \mbox{.} \]
	\end{Lemma}
	
	\begin{proof}
		Let $A_1, A_2, A_3 \subset \R^M$ be Lebesgue measurable sets. We have
		\begin{align*}
		\Vol_M(A_1 \cap A_2) - \Vol_M(A_1 \cap A_3) &\leq \Vol_M(A_1 \cap A_2) - \Vol_M(A_1 \cap A_2 \cap A_3) \\
		&= \Vol_M((A_1 \cap A_2)\setminus(A_1 \cap A_2 \cap A_3)) \\
		&\leq \Vol_M(A_2\setminus (A_2 \cap A_3))\\
		&= \Vol_M(A_2) - \Vol_M(A_2 \cap A_3) \mbox{.}
		\end{align*}
		Set now $A_1=W$, $A_2=\tilde{y}+W$ and $A_3=\tilde{z}+W$. Then
		\begin{align*}
		g_W(y)-g_W(z) &= \Vol_M(W \cap (\tilde{y} + W)) - \Vol_M(W \cap (\tilde{z} + W)) \\
		&\leq \Vol_M(\tilde{y} + W) - \Vol_M((\tilde{y} + W) \cap (\tilde{z} + W)) \\
		&= \Vol_M(W) - \Vol_M(W \cap (\tilde{y} - \tilde{z} + W)) \\
		&= g_W(0,\dots,0) - g_W(y-z) \mbox{.}
		\end{align*}
		Due to $g_W(z-y)=g_W(y-z)$, the same inequality holds for $g_W(z)-g_W(y)$ and, thereby, for the absolute value of $g_W(y)-g_W(z)$.
	\end{proof}
	
	\begin{Lemma}\label{t3}
		Let $W\subset\R^M$ be Lebesgue measurable and let $g_W$ be its generalized covariogram. Let $\tilde{y} \in \R^M$ and define $y\in \left(\R^M\right)^{M-1}$ by $y:=(\tilde{y},0,\dots,0)$. Then,
		\[ g_W(0,\dots,0) - g_W(y) = \frac{1}{2} \int \limits_{\R^M} |\mathds{1}_W(x+\tilde{y}) - \mathds{1}_W(x)| \diff x \mbox{.} \]
	\end{Lemma}
	
	\begin{proof}
		Using basic properties of the indicator function of a set, we have
		\begin{align*}
		&\int \limits_{\R^M} |\mathds{1}_W(x+\tilde{y}) - \mathds{1}_W(x)| \diff x = \int \limits_{\R^M} (\mathds{1}_W(x+\tilde{y}) - \mathds{1}_W(x))^2 \diff x \\
		&= \int \limits_{\R^M} \mathds{1}_W(x+\tilde{y})^2 \diff x + \int \limits_{\R^M} \mathds{1}_W(x)^2 \diff x - 2 \int \limits_{\R^M} \mathds{1}_W(x+\tilde{y}) \mathds{1}_W(x) \diff x \\
		&= 2 \Vol_M(W) - 2 \Vol_M(W \cap (\tilde{y} + W)) = 2 \left( g_W(0,\dots,0) - g_W(y) \right) \mbox{,}
		\end{align*}
		where the second to last equality follows from integrating $\mathds{1}_W(x)$ over $R^M+\tilde{y}$ instead of $\mathds{1}_W(x+\tilde{y})$ over $\R^M$ in the first integral.
	\end{proof}
	
	The following proposition will be used in the proof of \Cref{t10}. It relates the finiteness of the directional variation of a set $W$ in the direction $u\in \Sph^{M-1}$ to the existence of the directional derivate in $0$, in direction $u$, of the restriction of the generalized covariogram to a single argument, as well as to the Lipschitz continuity in direction $u$ of such a restrictions. The directional variation of $W$ in the direction $u$ can then be calculated as the Lipschitz constants of this restriction.
	
	\begin{Proposition}\label{t4}
		Let $W\subset\R^M$ be Lebesgue measurable, let $g_W$ be its generalized covariogram and let $u \in \Sph^{M-1}$. Define $y:=(u,0,\dots,0)$ and let $r\in \R$ with $r \neq 0$. The following statements are equivalent:
		\begin{enumerate}[(i)]
			\item $W$ has finite directional variation $\dVar_u(W)$,
			\item the derivative $\lim\limits_{r \rightarrow 0} \frac{g_W(0,\dots,0)-g_W(ry)}{|r|}$ exists and is finite,
			\item the function $g_{W}^u: r \mapsto g_W(ry)$ is Lipschitz.
		\end{enumerate}
		Additionally, the Lipschitz constant of $g_{W}^u$ is
		\[ \Lip(g_{W}^u) = \lim\limits_{r \rightarrow 0} \frac{g_W(0,\dots,0)-g_W(ry)}{|r|} = \frac{1}{2} \dVar_u(W) \mbox{.} \]
	\end{Proposition}
	
	\begin{proof}
		\Cref{t3} implies
		\[ \frac{g_W(0,\dots,0)-g_W(ry)}{|r|} = \frac{1}{2} \int \limits_{\R^M} \frac{|\mathds{1}_W(x+ru) - \mathds{1}_W(x)|}{|r|}\, \diff x \mbox{.} \]
		Applying \Cref{p2} with $f=\mathds{1}_W$, we obtain the equivalence of (i) and (ii) as well as the formula
		\[ \lim\limits_{r \rightarrow 0} \frac{g_W(0,\dots,0)-g_W(ry)}{|r|} = \frac{1}{2} \dVar_u(W) \mbox{.} \]
		
		We show now that (i) implies (iii). By \Cref{t2} we get for $r,s \in \R\setminus\{0\}$, that
		\begin{align*}
		|g_W(ry)-g_W(sy)| &\leq g_W(0,\dots,0) - g_W((r-s)y) \\
		&= \frac{1}{2} \int \limits_{\R^M} |\mathds{1}_W(x+(r-s)u) - \mathds{1}_W(x)|\, \diff x \\
		&= \frac{1}{2} |r-s| \int \limits_{\R^M} \frac{|\mathds{1}_W(x+(r-s)u) - \mathds{1}_W(x)|}{|r-s|}\, \diff x \\
		&\leq \frac{1}{2} \dVar_u(W) |r-s| \mbox{,}
		\end{align*}
		where the last inequality stems again from applying \Cref{p2} with $f=\mathds{1}_W$.
		Hence, $\Lip(g_{W}^u) \leq \frac{1}{2} \dVar_u(W)$.
		
		It remains to show that (iii) implies (i). For all $r \neq 0$ we have
		\[ \Lip(g_{W}^u) \geq \frac{g_W(0,\dots,0)-g_W(ry)}{|r|} = \frac{1}{2} \int \limits_{\R^M} \frac{|\mathds{1}_W(x+ru) - \mathds{1}_W(x)|}{|r|}\, \diff x \mbox{.} \]
		By \Cref{p2} the right-hand side converges to $\frac{1}{2} \dVar_u(W)$, as $r$ goes to $0$. Hence, $W$ has finite directional variation in the direction of $u$ and $\Lip(g_{W}^u) \geq \frac{1}{2} \dVar_u(W)$.
		
		Subsequently, (i) and (iii) are equivalent and
		\[ \Lip(g_{W}^u) = \frac{1}{2} \dVar_u(W) \]
		holds.
	\end{proof}
	
	\begin{Remark}\label{t5}
		Note that for $g_{W}^u$, i.e., the restriction of the generalized covariogram to the first argument along the direction $u\in \Sph^{M-1}$, the right derivative in $0$ can be expressed as:
		\[ (g_{W}^u)^\prime (0+) = \lim\limits_{r \rightarrow 0+} \frac{g_W(ry)-g_W(0,\dots,0)}{r} = -\lim\limits_{r \rightarrow 0} \frac{g_W(0,\dots,0)-g_W(ry)}{|r|} \mbox{.} \]
	\end{Remark}
	
	The second proposition of this section states the equivalence of the finiteness of the perimeter of a set $W$ and the existence of the right derivative in $0$ for the function $g_{W}^u$ introduced in \Cref{t4} and allows to determine the perimeter by means of integrating this right derivatives over the unit sphere. 
	
	\begin{Proposition}\label{t6}
	Let $W\subset\R^M$ be Lebesgue measurable and $g_W$ be its generalized covariogram. Define $y:=(u,0,\dots,0)$ and let $r \in \R$ with $r \neq 0$. The following two statements are equivalent:
		\begin{enumerate}[(i)]
			\item $W$ has finite perimeter $\Per(W)$,
			\item for all $u \in \Sph^{M-1}$ the derivative $(g_{W}^u)^\prime (0+) = \lim\limits_{r \rightarrow 0} \frac{g_W(ry)-g_W(0,\dots,0)}{r}$ exists and is finite.
		\end{enumerate}
		Additionally,
		\begin{equation}\label{e6}
		\Per(W)=-\frac{1}{\kappa_{M-1}} \int_{\Sph^{M-1}} (g_{W}^u)^\prime (0+) \mathcal{H}^{M-1}(du) \mbox{.}
		\end{equation}
	\end{Proposition}
	
	\begin{proof}
		\Cref{t4} and \Cref{t5} yield the identity
		\[ (g_{W}^u)^\prime (0+) = \lim\limits_{r \rightarrow 0+} \frac{g_W(ry)-g_W(0,\dots,0)}{r} = -\frac{1}{2} \dVar_u(W) \mbox{.} \]
		The equivalence of (i) and (ii), as well as \eqref{e6}, derive from applying \Cref{p1} with $f=\mathds{1}_W$ to this identity.
		
	\end{proof}
	
	\begin{Remark}\label{r1}
		It is known, that if $W\subset \R^M$ is a convex body, then $V_u(W)=2\mathcal{H}^{M-1}(W|u^\perp)$ holds for its directional variation, where $W|u^\perp$ is the orthogonal projection on the hyperplane with normal $u \in \Sph^{M-1}$. This result can be found in \cite[Eq. (10.1)]{S} and is restated in \cite{G}.
	\end{Remark}
	
	\section{Proof of the main results}
	
	The core idea of the proof of \Cref{t8} is to use the inequality
	\[ F_T^{(k)}(\xi_n) \leq N_T(\xi_n) \deg(\xi_n)^k \mbox{.} \]
	Since $N_T(\xi_n)= N_T(\xi_n)\, \mathds{1}(N_T(\xi_n) > 0) + N_T(\xi_n)\, \mathds{1}(N_T(\xi_n) = 0)$, we obtain by rearranging and taking the expectation, that
	\[ \E\left[ \deg(\xi_n)^k \right] \geq \E\left[ \frac{F_T^{(k)}(\xi_n) }{N_T(\xi_n)}\, \mathds{1}(N_T(\xi_n) > 0) \right] \mbox{.} \]
	To process this term it is necessary to ensure that the expectation of $N_T(\xi_n)$ is asymptotically bounded from above by a positive constant. This happens in two steps and is the content of \Cref{t10}. First, this expectation will be upper bounded in terms of $n$ and $T$, and, second, $T$ is chosen in such a way that the expectation indeed is asymptotically bounded from above by a constant. A possible choice of $T$ turns out to be $n^{-1/(M-1)}$ (In fact, this choice guarantees that the expectation is asymptotically also lower bounded by a positive constant, see \Cref{rem1}). This is also the step where the results about the generalized covariogram come into play.
	
	The above inequality implies that
	\[ \E\left[\deg (\xi_n)^k\right] \geq \frac{1}{K} \E \left[ F^{(k)}_T(\xi_n) \mathds{1}\left(0< N_T(\xi_n) \leq K \right) \right] \mbox{,} \]
	for all $K>0$.
	As will be shown in the proof of \Cref{t8}, further dissection of this term will make it necessary to find a lower bound on $\E\left[ \deg(x_1,\dots,x_M; \xi_{n-M}\cup \{ x_1,\dots,x_M \})^k \right]$. Here, $x_1,\dots,x_M$ are fixed points and $\xi_{n-M}$ denotes a random point set of $n-M$ points chosen uniformly and independently from the set $W$. This is elaborated in \Cref{t9} and gives the lower bound $cn^k$ for some constant $c>0$. A particular choice of $K$, namely, $K=2(M+1)\log n$ will then lead to the desired result.\\ \\
	The proof of the second part relies on the idea to choose a grid with mesh width $1/\sqrt[M]{n}$. The number of points contained in one of the cubes of this mesh, which are totally included in the convex body $W$, will contain a binomial distributed number of points of $\xi_n$, where each point has probability $1/n$ to lie within that cube. For large $n$, this number can be approximated by a Poisson random variable with rate $1$. It turns out that one can handle the Poisson($1$) random variables somehow similarly to the case where they are independent and show that, with probability going to $1$, for one of these cubes the contained set of random points has maximal degree. For the number of points $n$ going to infinity, this implies that also the degree of $\xi_n$ goes to infinity in probability.
	
	\begin{Proposition}\label{t10}
		Let $W\subset\R^M$, $M\geq 2$, be a convex body. For all $M \in \N$ and $T>0$ it holds that	
		\begin{equation}\label{eq8}
		\E\left[N_T(\xi_n)\right] \leq M\kappa_M^{M-1} \binom{n}{M} T^{M(M-1)} \Vol_M(W)(1+\mathcal{O}(T)),
		\end{equation}
		as $T \to 0$.
	\end{Proposition}
	
	\begin{proof}
		The first step consists in applying the binomial equivalent of the Mecke formula \eqref{e7} with $k=M$ and $f(x_1,\dots,x_M)=\mathds{1} \left( \exists i \in \mathcal{M}: \{x_1,\dots,x_M\} \subset \mathbb{B}^M_{T}\left(x_i\right) \right)$ to get
		\begin{align*}
		\E&\left[N_T(\xi_n)\right]  \\ &= \E \left[ \sum_{\{x_1,\dots,x_M\}\in {\xi_n \brack M}} \mathds{1} \left( \exists i \in \mathcal{M}: \{x_1,\dots,x_M\} \subset \mathbb{B}^M_{T}\left(x_i\right) \right) \right] \\
		&= \binom{n}{M} \int \limits_W \dots \int \limits_W \mathds{1} \left( \exists i \in \mathcal{M}: \{x_1,\dots,x_M\} \subset \mathbb{B}^M_{T}\left(x_i\right) \right) \diff x_1 \dots \diff x_M\\
		&\leq \sum_{i=1}^{M} \binom{n}{M} \int \limits_W \dots \int \limits_W \mathds{1} \left( \{x_1,\dots,x_M\} \subset \mathbb{B}^M_{T}\left(x_i\right) \right) \diff x_1 \dots \diff x_M \\
		&=M\binom{n}{M} \int \limits_{\R^M} \dots \int \limits_{\R^M} \left( \prod_{i=1}^{M-1} \mathds{1}\left( \| x_i \| \leq T \right) \right) \\ &~~~~~\times \int \limits_{\R^M} \mathds{1}(x_M \in W, x_M-x_1 \in W,\dots,x_M-x_{M-1} \in W)\, \diff x_M\, \diff x_1 \dots \diff x_{M-1} \\
		&=M\binom{n}{M} \int \limits_{\mathbb{B}^M_{T}(0)} \dots \int \limits_{\mathbb{B}^M_{T}(0)} g_W(x_1,\dots,x_{M-1})\, \diff x_1 \dots \diff x_{M-1} \\
		&=M\binom{n}{M} \int \limits_{0}^{T} \dots \int \limits_{0}^{T} \left( \prod_{i=1}^{M-1} r_i^{M-1} \right) \\ &~~~~~\times \int \limits_{\Sph^{M-1}} \dots \int \limits_{\Sph^{M-1}} g_W(r_1 u_1,\dots,r_{M-1} u_{M-1}) \mathcal{H}^{M-1}(\diff u_1) \dots \mathcal{H}^{M-1}(\diff u_{M-1})\,\diff r_1 \dots \diff r_{M-1} \mbox{.}
		\end{align*}
		For a fixed vector $u=(u_1,\dots,u_{M-1}) \in \left( \Sph^{M-1} \right)^{M-1}$ the Taylor expansion of the generalized covariogram $g_W(r_1 u_1,\dots,r_{M-1} u_{M-1})$, in $r=(r_1,\dots,r_{M-1})=(0,\dots,0)$, gives
		\begin{align*}
		g_W&(r_1 u_1,\dots,r_{M-1} u_{M-1}) \\
		&= g_W(0,\dots,0) + \sum_{i=1}^{M-1} r_i \frac{\partial}{\partial r_i} g_W(r_1 u_1,\dots,r_{M-1} u_{M-1})|_{r=(0,\dots,0)} + o(r_1,\dots,r_{M-1}) \mbox{.}
		\end{align*}
		Recall Equation \eqref{eq:perm}, i.e., the invariance of the covariogram $g_W$ under permutation of its arguments.
		We can conclude that integration of the Taylor expansion is possible since due to \eqref{eq:perm} and \Cref{t6} (ii) the partial derivative of $g_W(r_1 u_1,\dots,r_{M-1} u_{M-1})$ with respect to $r_i$ at $r=(r_1,\dots,r_{M-1})=0$ exists and is finite for each $i \in \{1,\dots,M-1\}$. This is done term by term. First,
		\begin{align*} 
		&\int \limits_{\Sph^{M-1}} \dots \int \limits_{\Sph^{M-1}} g_W(0,\dots,0) \mathcal{H}^{M-1}(\diff u_1) \dots \mathcal{H}^{M-1}(\diff u_{M-1}) \\ &= \Vol_M(W) \int \limits_{\Sph^{M-1}} \dots \int \limits_{\Sph^{M-1}} \mathcal{H}^{M-1}(\diff u_1) \dots \mathcal{H}^{M-1}(\diff u_{M-1}) \\
		&=(M \kappa_M)^{M-1} \Vol_M(W) \mbox{.}
		\end{align*}
		Second, by \Cref{t6} (iii), as well as \eqref{eq:perm}, one has for each $i \in \{1,\dots,M-1\}$,
		\begin{align*} 
		&\int \limits_{\Sph^{M-1}} \dots \int \limits_{\Sph^{M-1}} \frac{\partial}{\partial r_i} g_W(r_1 u_1,\dots,r_{M-1} u_{M-1})|_{r=(0,\dots,0)} \mathcal{H}^{M-1}(\diff u_1) \dots \mathcal{H}^{M-1}(\diff u_{M-1}) \\ &= - \kappa_{M-1} \mathcal{H}^{M-1}(\partial W) \int \limits_{\Sph^{M-1}} \dots \int \limits_{\Sph^{M-1}} \mathcal{H}^{M-1}(\diff u_1) \dots \mathcal{H}^{M-1}(\diff u_{M-1}) \\
		&=-\kappa_{M-1}(M \kappa_M)^{M-2} \Per(W) \mbox{.}
		\end{align*}
		
		Note that by \eqref{eq:perm}, for given $i \in \{1,\dots,M-1\}$ and $u_i \in \Sph^{M-1}$, $g_W(0,\dots,0,r_i u_i,0,\dots,0)=g_W(r_i u_i,0,\dots,0)=g_{W}^{u_i}(r_i)$ is a Lipschitz function in $r_i$. By \Cref{t4}, its Lipschitz constant is half of the bounded directional variation $V_{u_i}(W)$ of $W$ in the direction of $u_i$, which, by \Cref{r1}, is, due to the convexity of $W$, the Hausdorff measure $\mathcal{H}^{M-1}(W|u_i^\perp)$ of the orthogonal projection of $W$ onto the hyperplane $u_i^\perp$ orthogonal to $u_i$. This is used to show that
		\begin{align*}
		&|g_W(r_1 u_1,\dots,r_{M-1} u_{M-1})-g_W(0,\dots,0)| =g_W(0,\dots,0)-g_W(r_1 u_1,\dots,r_{M-1} u_{M-1}) \\
		&=\Vol_M\left( \bigcup_{i=1}^{M-1} \left( W \setminus (r_i u_i + W) \right) \right) \leq \sum_{i=1}^{M-1} \Vol_M\left( W\setminus(r_i u_i + W) \right) \\
		&=\sum_{i=1}^{M-1} \left( g_W(0,\dots,0) - g_W(0,\dots,0,r_i u_i,0,\dots,0) \right) \leq \sum_{i=1}^{M-1} \Lip\left(g_{W}^{u_i}\right)r_i \\
		&= \sum_{i=1}^{M-1} \mathcal{H}^{M-1}\left( W|u_i^\perp \right)r_i \mbox{.}
		\end{align*}
		Therefore, $\Vol_M(W) \geq g_W(r_1 u_1,\dots,r_{M-1} u_{M-1}) \geq \Vol_M(W) - \sum_{i=1}^{M-1} \mathcal{H}^{M-1}(W|u_i^\perp)r_i$ holds. Furthermore, by Cauchy's surface area formula \cite[Eq. (6.12)]{SW} it follows that
		\begin{align*}
		(M \kappa_M)^{M-1} \Vol_M(W) &\geq \int \limits_{\Sph^{M-1}}\dots\int \limits_{\Sph^{M-1}} g_W(r_1 u_1,\dots,r_{M-1} u_{M-1}) \mathcal{H}^{M-1}(\diff u_1) \dots \mathcal{H}^{M-1}(\diff u_{M-1}) \\
		&\geq (M \kappa_M)^{M-1} \Vol_M(W) - (M\kappa_M)^{M-2}\kappa_{M-1} \Per(W)\sum_{i=1}^{M-1} r_i \mbox{.}
		\end{align*}
		Thus, the $o(r_1,\dots,r_{M-1})$-term is positive and bounded by $(M\kappa_M)^{M-2}\sum_{i=1}^{M-1} r_i \kappa_{M-1} \Per(W)$ allowing us to infer the relation
		\begin{align*}
		&\int \limits_{\Sph^{M-1}} \dots \int \limits_{\Sph^{M-1}} g_W(r_1 u_1,\dots,r_{M-1} u_{M-1}) \mathcal{H}^{M-1}(\diff u_1) \dots  \mathcal{H}^{M-1}(\diff u_{M-1}) \\
		&=(M \kappa_M)^{M-2} \left( M\kappa_M \Vol_M(W) - \sum_{i=1}^{M-1} r_i \kappa_{M-1} \Per(W) \right) + o(r_1,\dots,r_{M-1})
		\end{align*}
		for all Lebesgue measurable sets $W$ with finite perimeter.
		 
		Therefore, one has that
		\begin{align*}
		& M \binom{n}{M} \int \limits_{0}^{T}\dots\int \limits_{0}^{T} \left( \prod_{i=1}^{M-1} r_i^{M-1} \right) \\ &~~~~~\times \int \limits_{\Sph^{M-1}}\dots\int \limits_{\Sph^{M-1}} g_W(r_1 u_1,\dots,r_{M-1} u_{M-1}) \mathcal{H}^{M-1}(\diff u_1) \dots \mathcal{H}^{M-1}(\diff u_{M-1})\,\diff r_1 \dots \diff r_{M-1} \\
		&=M \binom{n}{M} \int \limits_{0}^{T}\dots\int \limits_{0}^{T} (M \kappa_M)^{M-1} \Vol_M(W)\prod_{i=1}^{M-1} r_i^{M-1} \\
		&~~~-(M \kappa_M)^{M-2} \sum_{j=1}^{M-1} \kappa_{M-1} \Per(W) r_j^M \prod_{\substack{i=1 \\ i \neq j}}^{M-1} r_i^{M-1} + o(r_1^{M},\dots,r_{M-1}^{M})\, \diff r_1 \dots\diff r_{M-1} \\
		&=M \binom{n}{M} \left( (M\kappa_M)^{M-1} \frac{T^{M(M-1)}}{M^{M-1}} \Vol_M(W) \right. \\
		&\left.~~~- (M\kappa_M)^{M-2} \kappa_{M-1} \Per(W) \frac{(M-1)T^{M(M-1)+1}}{(M+1)M^{M-2}}  + o\left(T^{M^2-1}\right) \right) \\
		&=M \binom{n}{M} \left( \kappa_M^{M-1} \Vol_M(W)T^{M(M-1)} \right.\\
		&\left. ~~~- \frac{M-1}{M+1} \kappa_{M-1} \kappa_M^{M-2} \Per(W) T^{M(M-1)+1} +o\left(T^{M^2-1}\right) \right)\\
		&=M \kappa_M^{M-1} \binom{n}{M} T^{M(M-1)} \Vol_M(W) \left(1+\mathcal{O}\left(T\right)\right)
		\end{align*}
		Thus,
		\[
		\E\left[N_T(\xi_n)\right]  \leq M \kappa_M^{M-1} \binom{n}{M} T^{M(M-1)} \Vol_M(W) (1+\mathcal{O(T)}),
		\]
		as $T\to 0$, indeed describes the asymptotic behavior of the expectation of the functional $N_T(\xi_n)$ in dependence of $T$ and $M$.
	\end{proof}	
	As already mentioned, for the proof of \Cref{t8} it is essential to make sure that on average the functional $N_T(\xi_n)$ can be asymptotically bounded from above by a positive constant. The choice of $T=n^{-1/(M-1)}$ gives
	\begin{equation}\label{e8}
	\E\left[N_T(\xi_n)\right]  \leq M \kappa_M^{M-1} \binom{n}{M} n^{-M} \Vol_M(W) (1+\mathcal{O(T)}),
	\end{equation} 
	as $T\to 0$.
	\begin{Remark}\label{rem1}
		Note that, since 
		\[
			\mathds{1} \left( \{x_1,\dots,x_M\} \subset \mathbb{B}^M_{T}\left(x_1\right) \right)=1 \implies \mathds{1} \left( \exists i \in \mathcal{M}: \{x_1,\dots,x_M\} \subset \mathbb{B}^M_{T}\left(x_i\right) \right)=1,
		\]
		one can show with with the same proof, that also
		\[
			\E\left[N_T(\xi_n)\right]  \geq \kappa_M^{M-1} \binom{n}{M} T^{M(M-1)} \Vol_M(W) (1+\mathcal{O(T)})
		\]
		holds, as $T\to 0$, so that the choice of $T=n^{-1/(M-1)}$ in fact even guarantees that $\E[ N_T(\xi_n) ]$ behaves asymptotically like a constant.
	\end{Remark}	
	From here on we use the following convention: for a set of $M$ fixed points $\{x_1,\dots,x_M\}$, we denote by $\zeta$ the set $\zeta=\xi_{n-M}\cup\{x_1,\dots,x_M\}$.
	\begin{Proposition}\label{t9}
		Let $W\subset \R^M$, $M\geq 2$, be a convex body and let $\rho>0$. If there exists an $i \in \mathcal{M}$ such that $\{x_1,\dots,x_M\} \subset \mathbb{B}^M_{n^{-1/(M-1)}}(x_i)$ and $x_1,\dots,x_M \in \rho \mathbb{B}^M$, then,
		\[ \E\left[ \deg(x_1,\dots,x_M;\zeta)^k \right] \geq n^k \left(\Vol_M(W)\, \frac{\rho^{M-1} M!}{2^{M-1}} \left( 1 -  \exp\left( -\frac{2^{M-1} \rho}{M! \Vol_M(W)} \right) \right) \right)^k \]
		for all $k\in \N$ and sufficiently large $n$.
	\end{Proposition}

	\begin{proof}
		Let $\rho>0$ and let $x_1,\dots,x_M \in \rho \mathbb{B}^M$ be fixed vectors such that there exists an $i \in \mathcal{M}$ with $\{x_1,\dots,x_M\} \in \mathbb{B}^M_{n^{-1/(M-1)}}\left(x_i\right)$. Then,
		\begin{align*}
		&\E\left[ \deg(x_1,\dots,x_M;\zeta) \right] \\ 
		&= \E\left[ \sum_{z \in \xi_{n-M}} \mathds{1}\left(\zeta \cap \conv\{ x_1,\dots,x_m,z \} = \{ x_1,\dots,x_m,z \}\right) \right] \\
		&=(n-M)\Prob(\zeta^\prime \cap \conv\{ x_1,\dots,x_m,Y \} = \{ x_1,\dots,x_m,Y \} ) \mbox{,}
		\end{align*}
		where $\zeta^\prime=\xi_{n-M-1}\cup\{x_1,\dots,x_M,Y\}$ and $Y$ is a uniformly distributed random variable in W that is independent from $\xi_{n-M-1}$. This gives
		\begin{equation}\label{e12}
			\E\left[ \deg(x_1,\dots,x_M;\zeta) \right] = (n-M) \int_W \left(1-\frac{\Vol_{M}(x_1,\dots,x_M,y)}{\Vol_M(W)}\right)^{n-M-1}\, \diff y \mbox{.}
		\end{equation}
		
		Let $Q(x_1,\dots,x_M)$ be the $M$-dimensional cube with side length $\rho$, centered at $x_i \in \rho \mathbb{B}^M$, with one side parallel to the hyperplane spanned by $x_1,\dots,x_M$. Instead of integrating with respect to $y$ over $W$, the integration will be restricted to the cube $Q(x_1,\dots,x_M)$. Due to the fact that $x_1,\dots,x_M \in \mathbb{B}^M_{n^{-1/(M-1)}}\left(x_i\right)$, the pairwise distances between the points $x_1,\dots,x_M$ are less than $2n^{-1/(M-1)}$. Additionally, it holds that $\Vol_{M-1}(x_1,\dots,x_M)$ is smaller than the $(M-1)$-dimensional volume of the regular $M$-simplex with side length $2n^{-1/(M-1)}$, which can be estimated from above by $\frac{2^{M-1}}{(M-1)!}n^{-1}$. Hence, the estimate $\Vol_{M-1}(x_1,\dots,x_M) \leq \frac{2^{M-1}}{(M-1)!}n^{-1}$ can be used to get the lower bound
		\begin{align*}
		&\int \limits_W \left(1-\frac{\Vol_{M}(x_1,\dots,x_M,y)}{\Vol_M(W)}\right)^{n-M-1}\, \diff y \\ 
		&\geq \int \limits_{Q(x_1,\dots,x_M)} \left(1-\frac{\Vol_{M}(x_1,\dots,x_M,y)}{\Vol_M(W)}\right)^{n-M-1}\, \diff y \\ 
		&\geq \int \limits_{0}^{\rho}\dots\int \limits_{0}^{\rho} \left(1-\frac{\Vol_{M-1}(x_1,\dots,x_M)}{M \Vol_M(W)}\,y_M\right)^{n-M-1}\, \diff y_M \dots \diff y_1 \\ 
		&= \rho^{M-1} \int \limits_{0}^{\rho} \left(1-\frac{\Vol_{M-1}(x_1,\dots,x_M)}{M \Vol_M(W)}\,y_M\right)^{n-M-1}\, \diff y_M \\ 
		&= \rho^{M-1}\, \frac{M \Vol_M(W)}{\Vol_{M-1}(x_1,\dots,x_M)} \int \limits_{0}^{\frac{\Vol_{M-1}(x_1,\dots,x_M)\rho}{M\Vol_M(W)}} \left(1-t\right)^{n-M-1}\, \diff t \\ 
		&=  \frac{\rho^{M-1} M \Vol_M(W)}{\Vol_{M-1}(x_1,\dots,x_M)} \frac{1}{n-M} \left( 1 -  \left(1-\frac{\Vol_{M-1}(x_1,\dots,x_M) \rho}{M \Vol_M(W)}\right)^{n-M} \right) \\ 
		& \geq \Vol_M(W)\,\frac{\rho^{M-1} n M!}{2^{M-1} (n-M)} \left( 1 -  \left(1-\frac{2^{M-1} \rho}{n M! \Vol_M(W)}\right)^{n-M} \right) \\ 
		&\geq \Vol_M(W)\,\frac{\rho^{M-1} n M!}{2^{M-1} (n-M)} \left( 1 -  \exp\left( -\frac{(n-M) 2^{M-1} \rho}{n M! \Vol_M(W)} \right) \right) \\ 
		& \geq \Vol_M(W)\,\frac{\rho^{M-1} n M!}{2^{M-1} (n-M)} \left( 1 -  \exp\left( -\frac{2^{M-1} \rho}{M! \Vol_M(W)} \right) \right)
		\end{align*}
		for large enough $n$. Combining this result with
		\[ \E\left[ \deg(x_1,\dots,x_M;\zeta)^k \right] \geq (n-M)^k \left( \int_W \left(1-\frac{\Vol_{M}(x_1,\dots,x_M,y)}{\Vol_M(W)}\right)^{n-M-1}\, \diff y \right)^k \mbox{,} \]
		which follows from \eqref{e12} by applying Jensen's inequality, finishes the proof.
	\end{proof}
	
	\begin{proof}[Proof of \Cref{t8} (i)]
		Clearly, the inequality
		\begin{equation}\label{e1}
		F^{(k)}_T(\xi_n) \leq N_T(\xi_n)\deg (\xi_n)^k 
		\end{equation}
		holds for the set $\xi_n$, and, by $N_T(\xi_n)=N_T(\xi_n)\,\mathds{1}(N_T(\xi_n)>0) + N_T(\xi_n)\,\mathds{1}(N_T(\xi_n)=0)$, also
		\begin{equation}\label{e2}
		\deg(\xi_n)^k \geq \frac{F^{(k)}_T(\xi_n)}{N_T(\xi_n)}\,\mathds{1}(N_T(\xi_n)>0)
		\end{equation}
		holds.
		
		Since $\deg\xi_n$ is invariant under non-degenerate affine transformations we can apply such a transformation to $W$ and $\E[\deg\xi_n]$ will not change. By John's Theorem \cite[Thm.~10.12.2]{S}, there exists an ellipsoid $E$ such that $E \subset W \subset M E$. First, we apply the affine transformation so that the area of $W$ becomes equal to one, making the Lebesgue measure coincide with the probability measure defining $\xi_n$. Second, we apply a volume preserving affine transformation that carries $E$ to  $r \mathbb{B}$ and consequently $ME$ to $Mr\mathbb{B}$. From now on, let $W$ be in this position and assume that $r\mathbb{B} \subset W \subset Mr\mathbb{B}$ holds.
		
		Hence, for a random set $\xi_n$ of $n$ points chosen uniformly and independently from $W$, it holds that
		\begin{equation}\label{e3}
		\E \left[ \deg (\xi_n)^k \right] \geq \E \left[ \frac{F^{(k)}_T(\xi_n)}{N_T(\xi_n)}\,\mathds{1}(N_T(\xi_n)>0) \right] \mbox{.}
		\end{equation}
		
		Intuitively, one would expect an $M$-element subset of points $\{x_1,\dots,x_M\}\subset\xi_n$ to be of highest degree, if it is the $M$-element subset where the points are the closest to each other. As a notion of closeness the existence of a point $x_i\subset\{x_1,\dots,x_M\}$, $i \in \mathcal{M}$, for which all the remaining points lie in a ball of certain radius, centered at $x_i$, is used, as can be deduced from the definition of the functionals $N_T(\xi_n)$ and $F_T(\xi_n)$.
		
		Since $N_T(\xi_n)$ counts the number of $M$-element subsets that satisfy the above mentioned closeness for a radius $T$, it has to be made sure that this dependence on $T$ is chosen correctly. Due to \Cref{t10}, the choice of $T=n^{-1/(M-1)}$ makes sure that $\E[N_T(\xi_n)]$ asymptotically behaves like a positive constant and allows the determination of an upper bound for $\E \left[ \deg \xi_n \right]$.

		Equation \eqref{e3} can be broken down further, for all $K>0$, into
		\begin{equation}\label{e9}
		\E\left[\deg (\xi_n)^k\right] \geq \frac{1}{K} \E \left[ F^{(k)}_T(\xi_n) \mathds{1}\left(0< N_T(\xi_n) \leq K \right) \right] \mbox{.}
		\end{equation}
		Now, by exploiting the definition of $F_T^{(k)}$, the linearity of the expectation and the independence of the points of $\xi_n$, the right-hand side of \eqref{e9} gives
		\begin{equation}\label{e10}
		\begin{split}
		&\frac{1}{K} \E \left[ F^{(k)}_T(\xi_n) \mathds{1}\left(0< N_T(\xi_n) \leq K \right) \right] \\
		&=\frac{\binom{n}{M}}{K} \E \left[ \mathds{1} \left( \exists i \in \mathcal{M}: \{x_1,\dots,x_M\} \subset \mathbb{B}^M_{T}\left(x_i\right) \right) \deg(x_1,\dots,x_M;\xi_n)^k\, \mathds{1}\left(0< N_T(\xi_n) \leq K \right) \right]  \\ 
		&= \frac{\binom{n}{M}}{K} \int \limits_{W} \dots \int \limits_{W} \mathds{1} \left( \exists i \in \mathcal{M}: \{x_1,\dots,x_M\} \subset \mathbb{B}^M_{T}\left(x_i\right) \right)\\ 
		&~~~\times\E\left[ \deg(x_1,\dots,x_M;\zeta)^k \mathds{1}\left(0< N_T(\zeta) \leq K \right) \right] \diff x_1 \dots \diff x_M \mbox{.}
		\end{split}
		\end{equation}
		From the elementary equality $\E[X \mathds{1}(A)] = \E[X] - \E[X \mathds{1}(A^c)] $, which holds for any random variable $X$ and any event $A$, \eqref{e10} can be bounded in the following way:
		\begin{equation}\label{e11}
		\begin{split}
		&\E\left[ \deg(x_1,\dots,x_M;\zeta)^k \mathds{1}\left(0< N_T(\zeta) \leq K \right) \right] \\
		&=\E\left[ \deg(x_1,\dots,x_M;\zeta)^k \right] - \E\left[ \deg(x_1,\dots,x_M;\zeta)^k \mathds{1}\left( N_T(\zeta) \in \{0\} \cup (K,\infty) \right) \right] \\
		&\geq \E\left[ \deg(x_1,\dots,x_M;\zeta)^k \right] - n^k\E\left[ \mathds{1}\left( N_T(\zeta) \in \{0\} \cup (K,\infty) \right) \right] \\
		&= \E\left[ \deg(x_1,\dots,x_M;\zeta)^k \right] - n^k\Prob\left( N_T(\zeta) \in \{0\} \cup (K,\infty) \right)\\
		&=\E\left[ \deg(x_1,\dots,x_M;\zeta)^k \right] - n^k\Prob\left( N_T(\zeta) > K \right) - n^k\Prob\left( N_T(\zeta) = 0 \right).
		\end{split}
		\end{equation}
		
		Let $\rho = r/2$. Then, the ball $\rho \mathbb{B}^M$ is contained in $r \mathbb{B}^M \subset W$ and every point of $\rho \mathbb{B}^M$ is farther than $\rho = r/2$ away from the boundary of $W$. Then,
		\begin{align*}
		N_T(\zeta) &\leq N_T(\xi_{n-M}) + \sum_{i=1}^{M} |\xi_{n-M} \cap \mathbb{B}^M_{T}(x_i)| + 1 \\ 
		& \leq N_T(\xi_{n-M}) + M N_{2T}(\xi_{n-M}) + 1 \\ 
		& \leq (M+1) N_{2T}(\xi_{n-M}) + 1 \\
		& \leq (M+1) N_{2T}(\xi_n) + 1 \mbox{,}
		\end{align*}
		where the fact was used that, if $M$ points lie in $\mathbb{B}^M_T(x_i)$, then their pairwise distance is at most $2T$. This implies, with $K_n=2 (M+1) \ln n$, that
		\[ \Prob(N_T(\zeta) \geq K_n) \leq \Prob((M+1)N_{2T}(\xi_n)+1 \geq K_n) \leq \Prob(N_{2T}(\xi_n)\geq  \ln n) \mbox{.} \]
		Setting $T=n^{-1/(M-1)}$, this yields, together with Markov's inequality and \Cref{t10}, that
		\[ \Prob(N_T(\xi_n) \geq K_n) \leq \Prob(N_{2T}(\xi_n)\geq  \ln n) \leq \frac{\E\left[ N_{2T}(\xi_n)  \right]}{\ln n} \leq \frac{c}{\ln n} \mbox{,}  \]
		for some constant $c>0$. 
		
		Furthermore, one sees that $\mathds{1} \left( \exists i \in \mathcal{M}: \{x_1,\dots,x_M\} \subset \mathbb{B}^M_{T}\left(x_i\right) \right) = 1$ implies that $N_T(\zeta)\geq 1$, and subsequently that $\Prob(N_T(\zeta)=0)=0$. Plugging these two results, as well as the statement of \Cref{t9}, into \Cref{e10}, one obtains
		\begin{align*}
		\frac{1}{K_n}& \E \left[ F^{(k)}_T(\xi_n) \mathds{1}\left( 0 <N_T(\xi_n) \leq K_n \right) \right] \\ 
		&\qquad\qquad\geq \frac{\binom{n}{M}}{2(M+1)\ln n} \left[ n^k \left(\frac{\rho^{M-1} M!}{2^{M-1}} \left( 1 -  \exp\left( -\frac{2^{M-1} \rho}{M!} \right) \right)  \right)^k- c \frac{n^k}{\ln n}  \right] \\ 
		&\qquad\qquad\quad \times \int \limits_{(\rho \mathbb{B}^M)^M} \mathds{1} \left( \exists i \in \mathcal{M}: \{x_1,\dots,x_M\} \subset \mathbb{B}^M_{T}\left(x_i\right) \right) \diff x_1 \dots \diff x_M \mbox{.}
		\end{align*}
		The value of the integral is $c/n^M$, for some constant $c>0$. Therefore, one concludes that
		\[ \E\left[ \deg \left( \xi_n \right)^k \right]  \geq c ~\frac{n^k}{\ln n} \mbox{,} \]
		for $n$ large enough, with some constant $c>0$.
		
	\end{proof}
	
	\begin{proof}[Proof of \Cref{t8} (ii)]
		This part of the proof barely differs from the proof for the degree in $2$-dimensional case as elaborated in \cite{BMR}, so that only the minor differences will be pointed out. Namely, instead of introducing a grid with mesh width $1/\sqrt{n}$ in the plane, one has to use a grid with mesh width $1/\sqrt[M]{n}$ in $\R^M$, and instead of considering squares one has to consider $M$-dimensional cubes. The proof then follows exactly as in the $2$-dimensional case, except for the obvious fact that one has to use $n-M$ for every appearance of $n-2$.
	\end{proof}	
		
	\begin{Remark}
		Upon careful examination of the proofs of \Cref{t10} and \Cref{t9} one sees that the result of \Cref{t8} (i) can be extended further. Namely, both propositions hold, with different constants, if one chooses a distribution over $W$ which has a density that can be bounded from above and below by positive constants. One only has to incorporate the density of this distribution into the integrals in \eqref{e7} and \eqref{e12}, respectively. The upper bound of the density can then be used to bound the integral in \eqref{e7} from above, whereas the lower bound of the density gives a bound from below for the integral in \eqref{e12}.
	\end{Remark}
		
	\subsection*{Acknowledgments}
		
	The author would like to thank Christoph Th\"ale and Julian Grote for helpful discussion concerning the topics of this paper. Furthermore, the author expresses his gratitude towards the referees for their suggestions regarding improvements of the paper.

\end{document}